\documentclass[12pt]{amsart}
\usepackage[utf8]{inputenc}
\usepackage{fullpage}
\usepackage{graphicx}
\usepackage{amsmath, amsthm}
\usepackage{amsfonts}
\usepackage{amssymb}
\usepackage{url}

\usepackage[skins,breakable]{tcolorbox}

\usepackage[backrefs]{amsrefs}

\newtheorem{theorem}{Theorem}
\newtheorem{lemma}{Lemma}

\newtheorem{question}{Question}
\newtheorem{example}{Example}
\newtheorem{defn}{Definition}

\usepackage{hyperref}
\usepackage{tikz}
\usetikzlibrary{positioning}

\def\diam{\text{diam}}
\keywords{Zero forcing, multicolor forcing}
\title{Multi-color forcing in graphs}
\author{Chassidy Bozeman}
\author{Pamela E. Harris}\thanks{P. E. Harris was supported in part by the National Science Foundation grant DMS-1620202.}
\author{Neel Jain}
\author{Ben Young}
\author{Teresa Yu}
\address[C. Bozeman]{Department of Mathematics and Statistics, Mount Holyoke College, South Hadley, MA 01075, USA}
\email{\textcolor{blue}{\href{mailto:cbozeman@mtholyoke.edu}{cbozeman@mtholyoke.edu}}}
\address[P. E. Harris, N. Jain, B. Young, and T. Yu]{Department of Mathematics and Statistics, Williams College,
Williamstown, MA 01267, USA} 
\email{\textcolor{blue}{\href{mailto:peh2@williams.edu}{peh2@williams.edu}}, \textcolor{blue}{\href{mailto:nsj2@williams.edu}{nsj2@williams.edu}}, \textcolor{blue}{\href{mailto:bly1@williams.edu}{bly1@williams.edu}}, and \textcolor{blue}{\href{mailto:twy1@williams.edu}{twy1@williams.edu}}}

\begin{document}
\date{}

\begin{abstract}
Let $G=(V,E)$ be a finite connected graph along with a coloring of the vertices of $G$ using the colors in a given set $X$.
In this paper, we introduce multi-color forcing, a generalization of zero-forcing on graphs, and give conditions in which the multi-color forcing process terminates regardless of the number of colors used. We give an upper bound on the number of steps required to terminate a forcing procedure in terms of the number of vertices in the graph on which the procedure is being applied. We then focus on multi-color forcing with three colors and analyze the end states of certain families of graphs, including complete graphs, complete bipartite graphs, and paths, based on various initial colorings. We end with a few directions for future research.
\end{abstract}

\maketitle
\section{Introduction}

Let $G=\left(V,E\right)$ be a connected graph with all vertices colored blue or white.
If $v$ is a blue vertex of $G$, then $v$ {\em forces} a white vertex $u$ to blue if and only if $u$ is the only white vertex in the neighborhood of $v$. This procedure is called the {\em color change rule} \cite{AIMMINIMUMRANKSPECIALGRAPHSWORKGROUP20081628}. Then the {\em zero forcing} process is the procedure of applying the color change rule until no more changes are possible. If $S$ is an initial set of vertices that is colored blue such that the entire graph can be colored blue by applying the color change rule, then $S$ is a {\em zero forcing} set. A main goal in the study of zero forcing  on graphs is to determine the the minimum size of a zero forcing set, known as the {\em zero forcing  number} of a graph.

The zero forcing procedure was introduced in linear algebra as a tool for studying the maximum nullity over a family of matrices \cite{AIMMINIMUMRANKSPECIALGRAPHSWORKGROUP20081628}, and independently in physics, computer science, and network science \cite{PhysRevLett.99.100501,DREYER20091615,BARIOLI2010401}. Since its introduction it has motivated the research of many mathematicians \cite{JGT:JGT21637,edholm,HUANG20102961}.
Kalinowski, Kam\v{c}ev, and Sukakov considered  bipartite, random, and pseudorandom graphs and established bounds for the zero forcing  number of these graphs \cite{Kalinowski}. Chilakamarri, Dean, Kang, and Yi defined the \textit{iteration index} of a graph to be the number of implementations of the color change rule, such that all vertices of a graph are blue. In their work, Chilakamarri et. al. determined that the minimum of iteration indices of all minimum zero forcing   sets of $G$ was a graph invariant, which they called the \emph{iteration index} of $G$ \cite{Chilakamarri}.
Hogben, Huynh, Kingsley, Meyer, Walker, and Young characterized graphs with extreme minimum propagation times (iteration index) and showed that the diameter is an upper bound for the propagation time for a tree, but that in general the diameter of a graph can be arbitrarily larger than its minimum propagation time \cite{Hogben}.

Our research is motivated by the following open problem posed by Daniela Ferrero and listed on the American Institute of Mathematics  website for the recent conference on zero forcing  and its applications~\cite{AIMOpen}.  
\begin{center}
\begin{tcolorbox}
\begin{minipage}{\textwidth}
{\bf{Problem 1.58}.}
 \emph{What is the generalization to zero forcing   with multiple colors? Maybe the colors are linearly ordered, i.e., red can force purple, purple can force blue, etc. Maybe the graph has an underlying coloring and filled vertices force according to the zero forcing   rule if they are on their preferred color and don't otherwise. This was an application of an ecosystem where colors represent animals (or bacteria etc.) and the underlying coloring corresponds to an animals suitable habitat."}
\end{minipage}
\end{tcolorbox}
\end{center}

In this paper, we extend the graph theoretical idea of forcing on a graph with two colors to forcing on a graph with multiple colors. In Section \ref{section:terminology}, we introduce the terminology used throughout the paper. Termination of multi-forcing processes is discussed in Section \ref{section:termination}, and we give a technique for constructing the end state of a graph $G$ from the end state of a smaller graph $G'$. In Section \ref{section:cyclic}, we focus specifically on multi-color forcing with three colors. Lastly, in Section \ref{sec:future} we end with a few directions for future research.

\section{Multi-color forcing terminology}\label{section:terminology}
Throughout this paper, $G$ is a finite simple graph, and if $n\in\mathbb{N}:=\{1,2,3,\ldots\}$, we let $[n]=\{1,2,\ldots,n\}$. For a set of ``colors'' $X\subseteq[n]$, we say an {\em $X-$colored graph} is a graph whose vertices are colored using colors from $X$, and a {\em forcing network} is a pair $(X,R)$ where $R$ is an ordered set of color change rules corresponding to the colors in $X$.

The simplest way to visually understand a forcing network is by a directed graph. For example, let $X=\{1,2,3\}$ and $R=\{1\rightarrow 2, 2\rightarrow 3, 3\rightarrow 1\}.$ We can represent this network visually by the directed graph in Figure \ref{fig: Cyclic Forcing Network w/ Three Colors}. Although it is not evident from the visual representation of a forcing network, we apply color change rules in the order in which they appear in $R$. 

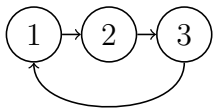
\begin{figure}[ht!]
    \centering
    \begin{tikzpicture}[-latex ,auto ,node distance =2 cm and 1cm ,on grid ,
    semithick ,
    state/.style ={ circle ,top color =white,
    draw,  minimum width =.5cm}]
    \node[state] (A)
    {$1$};
    \node[state] (B) [right=of A] {$2$};
    \draw [->] (A) edge (B);
    \node[state] (C) [right =of B] {$3$};
    \draw[->] (B) edge (C);
    \draw [->] (C) to [out=270,in=270] (A);
    \end{tikzpicture}\\
    \caption{Cyclic forcing color network with three colors.}
    \label{fig: Cyclic Forcing Network w/ Three Colors}
\end{figure}

In our generalization of zero forcing, we do not allow different color change rules to occur simultaneously, and we do allow vertices to force their neighbors, subject to the rules in $R$, even if they have multiple neighbors of a given color. It should be noted that the condition that a blue vertex can force a neighboring white vertex if and only if it has exactly one neighboring white vertex is necessary in order to make the connection of zero forcing to linear algebra as in \cite{AIMMINIMUMRANKSPECIALGRAPHSWORKGROUP20081628}. Although relaxing this condition in multi-color forcing loses the linear algebra connection, it gives rise to interesting graph theoretical problems in its own right.  

A \emph{forcing step}, denoted FS, occurs in a graph $G$ when a color change rule from $R$ is applied to $G$. A \emph{propagating forcing step}, denoted PFS,  occurs when a given color change rule is applied  until no more forces are possible using this rule. This means that a forced vertex can then force other vertices under the given rules, all within a fixed propagating forcing step. This is called {\em forcing with propagation} and is illustrated in the following.

\begin{example}\label{forcing with prop}{\rm 
Consider the forcing network $(X,R)$ where $X=\{1,2,3\}$ and $R=\{1\to 2, 2\to 3, 3\to 1\}$, as in Figure \ref{fig: Cyclic Forcing Network w/ Three Colors}. We apply this network to the $X-$colored graph $G$ shown in Figure~\ref{no termination}. The forcing process terminates after 4 forcing steps, but after only 2 propagating forcing steps.

\begin{figure}[ht!]     
\centering
\resizebox{\textwidth}{!}{
    \begin{tikzpicture}[-latex,auto ,node distance =2 cm and 1cm ,on grid ,
    semithick ,
    state/.style ={ circle ,top color =white,
    draw,  minimum width =.5cm}]
    \node[state] at (-1,-1/2) (A){$1$};
    \node[state] at (0,-1) (B) {$3$};
    \node[state] at (1,-1/2) (C) {$3$};
    \node[state] at (1,1/2) (D) {$2$};
    \node[state] at (-1,1/2) (F) {$1$};
    \node[state, label = {$G$}] at (0,1) (E) {$2$};
    \draw [-] (A) -- (B); 
    \draw [-] (B) -- (C);
    \draw [-] (C) -- (D);
    \draw [-] (D) -- (E);
    \draw [-] (E) -- (F);
    \draw [-] (A) -- (F);
    \end{tikzpicture}
\qquad
    \begin{tikzpicture}[-latex ,auto ,node distance =2 cm and 1cm ,on grid ,
    semithick ,
    state/.style ={ circle ,top color =white,
    draw,  minimum width =.5cm}]
    \node[state] at (-1,-1/2) (A){$1$};
    \node[state] at (0,-1) (B) {$3$};
    \node[state] at (1,-1/2) (C) {$3$};
    \node[state] at (1,1/2) (D) {$2$};
    \node[state] at (-1,1/2) (F) {$1$};
    \node[state, label = {$G$ after FS 1}] at (0,1) (E) {$1$};
    \draw [-] (A) -- (B); 
    \draw [-] (B) -- (C);
    \draw [-] (C) -- (D);
    \draw [-] (D) -- (E);
    \draw [-] (E) -- (F);
    \draw [-] (A) -- (F);
    \end{tikzpicture}
\qquad
    \begin{tikzpicture}[-latex ,auto ,node distance =2 cm and 1cm ,on grid ,
    semithick ,
    state/.style ={ circle ,top color =white,
    draw,  minimum width =.5cm}]
    \node[state] at (-1,-1/2) (A){$1$};
    \node[state] at (0,-1) (B) {$3$};
    \node[state] at (1,-1/2) (C) {$3$};
    \node[state] at (1,1/2) (D) {$1$};
    \node[state] at (-1,1/2) (F) {$1$};
    \node[state, label = {$G$ after FS2  (PFS1)}] at (0,1) (E) {$1$};
    \draw [-] (A) -- (B); 
    \draw [-] (B) -- (C);
    \draw [-] (C) -- (D);
    \draw [-] (D) -- (E);
    \draw [-] (E) -- (F);
    \draw [-] (A) -- (F);
    \end{tikzpicture}
\qquad
     \begin{tikzpicture}[-latex ,auto ,node distance =2 cm and 1cm ,on grid ,
    semithick ,
    state/.style ={ circle ,top color =white,
    draw,  minimum width =.5cm}]
    \node[state] at (-1,-1/2) (A){$3$};
    \node[state] at (0,-1) (B) {$3$};
    \node[state] at (1,-1/2) (C) {$3$};
    \node[state] at (1,1/2) (D) {$3$};
    \node[state] at (-1,1/2) (F) {$1$};
    \node[state, label = {$G$ after FS3}] at (0,1) (E) {$1$};
    \draw [-] (A) -- (B); 
    \draw [-] (B) -- (C);
    \draw [-] (C) -- (D);
    \draw [-] (D) -- (E);
    \draw [-] (E) -- (F);
    \draw [-] (A) -- (F);
    \end{tikzpicture}
\qquad
    \begin{tikzpicture}[-latex ,auto ,node distance =2 cm and 1cm ,on grid ,
    semithick ,
    state/.style ={ circle ,top color =white,
    draw,  minimum width =.5cm}]
    \node[state] at (-1,-1/2) (A){$3$};
    \node[state] at (0,-1) (B) {$3$};
    \node[state] at (1,-1/2) (C) {$3$};
    \node[state] at (1,1/2) (D) {$3$};
    \node[state] at (-1,1/2) (F) {$3$};
    \node[state, label = {$G$ after FS 4  (PFS2)}] at (0,1) (E) {$3$};
    \draw [-] (A) -- (B); 
    \draw [-] (B) -- (C);
    \draw [-] (C) -- (D);
    \draw [-] (D) -- (E);
    \draw [-] (E) -- (F);
    \draw [-] (A) -- (F);
    \end{tikzpicture}
    }
    \caption{Example of applying at network to an $X-$colored graph.}\label{no termination} 
\end{figure}
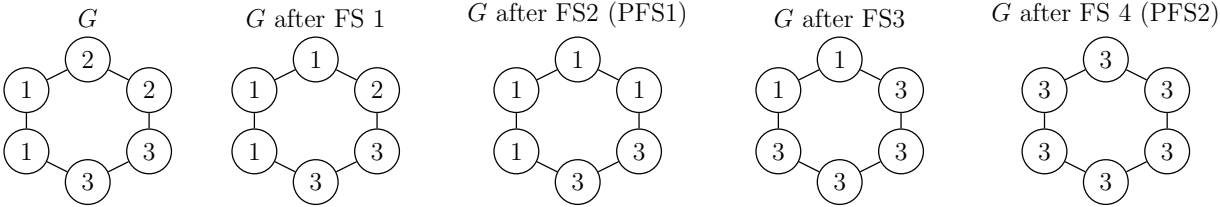
}

\end{example}

 For an $X-$colored graph $G$ and $i\geq 0$, we use $\ell_{i}(G)$ to denote the state of the coloring on $G$ after the $i^{\text{th}}$ propagating forcing step, while $\epsilon(\ell_0(G))$ denotes the end state of $G$ arising from the initial coloring $\ell_{0}(G)$ . We use $\epsilon(\ell_0(G))=\vec{c}$ to denote that the end state of $G$ has all vertices colored $c$. 
\section{Termination of the multi-color forcing procedure}\label{section:termination}

A natural question to ask is whether or not a given forcing network terminates. In this section, we show that a forcing network applied to a graph on $n$ vertices terminates in at most $n-1$ steps when forcing with propagation occurs. We also give a condition that will guarantee a forcing network will still terminate even when forcing with propagation is not required. In general, a forcing network might not terminate if forcing with propagation is not required. The following example demonstrates this.

\begin{example}\label{ex: does not terminate}{\rm 
Consider the forcing network $(X,R)$ where $X=\{1,2,3\}$ and $R=\{1\to 2, 2\to 3, 3\to 1\}$, as in Example \ref{forcing with prop}. When applying this network to $G$ without propagation, after three steps the graph coloring is equivalent to its original state, which demonstrates that the process will not terminate. We illustrate this in Figure \ref{fig:no termination}.
\begin{figure}[ht!]
    \centering
   \begin{tikzpicture}[-latex ,auto ,node distance =2 cm and 1cm ,on grid ,
    semithick ,
    state/.style ={ circle ,top color =white,
    draw,  minimum width =.5cm}]
    \node[state] at (-1,-1/2) (A){$1$};
    \node[state] at (0,-1) (B) {$3$};
    \node[state] at (1,-1/2) (C) {$3$};
    \node[state] at (1,1/2) (D) {$2$};
    \node[state] at (-1,1/2) (F) {$1$};
    \node[state, label = {$G$}] at (0,1) (E) {$2$};
    \draw [-] (A) -- (B); 
    \draw [-] (B) -- (C);
    \draw [-] (C) -- (D);
    \draw [-] (D) -- (E);
    \draw [-] (E) -- (F);
    \draw [-] (A) -- (F);
    \end{tikzpicture}
    \qquad
    \begin{tikzpicture}[-latex ,auto ,node distance =2 cm and 1cm ,on grid ,
    semithick ,
    state/.style ={ circle ,top color =white,
    draw,  minimum width =.5cm}]
    \node[state] at (-1,-1/2) (A){$1$};
    \node[state] at (0,-1) (B) {$3$};
    \node[state] at (1,-1/2) (C) {$3$};
    \node[state] at (1,1/2) (D) {$2$};
    \node[state] at (-1,1/2) (F) {$1$};
    \node[state, label = {$G$ after FS 1}] at (0,1) (E) {$1$};
    \draw [-] (A) -- (B); 
    \draw [-] (B) -- (C);
    \draw [-] (C) -- (D);
    \draw [-] (D) -- (E);
    \draw [-] (E) -- (F);
    \draw [-] (A) -- (F);
    \end{tikzpicture} 
    \qquad
    \begin{tikzpicture}[-latex ,auto ,node distance =2 cm and 1cm ,on grid ,
    semithick ,
    state/.style ={ circle ,top color =white,
    draw,  minimum width =.5cm}]
    \node[state] at (-1,-1/2) (A){$1$};
    \node[state] at (0,-1) (B) {$3$};
    \node[state] at (1,-1/2) (C) {$2$};
    \node[state] at (1,1/2) (D) {$2$};
    \node[state] at (-1,1/2) (F) {$1$};
    \node[state, label = {$G$ after FS2}] at (0,1) (E) {$1$};
    \draw [-] (A) -- (B); 
    \draw [-] (B) -- (C);
    \draw [-] (C) -- (D);
    \draw [-] (D) -- (E);
    \draw [-] (E) -- (F);
    \draw [-] (A) -- (F);
    \end{tikzpicture} 
    \qquad
    \begin{tikzpicture}[-latex ,auto ,node distance =2 cm and 1cm ,on grid ,
    semithick ,
    state/.style ={ circle ,top color =white,
    draw,  minimum width =.5cm}]
    \node[state] at (-1,-1/2) (A){$3$};
    \node[state] at (0,-1) (B) {$3$};
    \node[state] at (1,-1/2) (C) {$2$};
    \node[state] at (1,1/2) (D) {$2$};
    \node[state] at (-1,1/2) (F) {$1$};
    \node[state, label = {$G$ after FS3}] at (0,1) (E) {$1$};
    \draw [-] (A) -- (B); 
    \draw [-] (B) -- (C);
    \draw [-] (C) -- (D);
    \draw [-] (D) -- (E);
    \draw [-] (E) -- (F);
    \draw [-] (A) -- (F);
    \end{tikzpicture}
    \caption{Graph $G$ with non-propagation forcing steps that does not terminate.}\label{fig:no termination}
\end{figure}
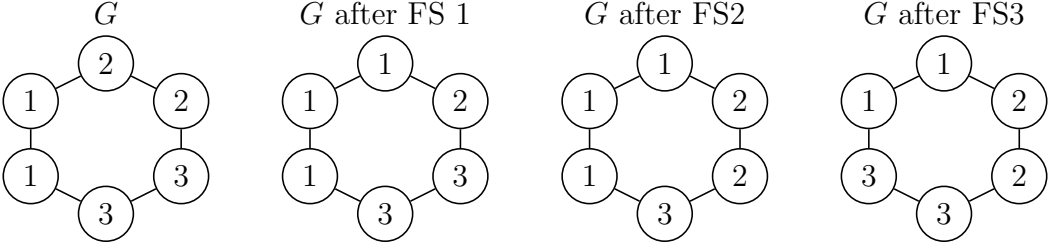

However, when applying the rules  $1\rightarrow 2$, $2\rightarrow 3$, and $3 \rightarrow 1$ in this order and with propagation, the process terminates with all vertices colored $3$, as demonstrated in Example~\ref{forcing with prop}.  
}
\end{example}
It should be noted that when a forcing network terminates, it need not terminate with all vertices colored the same color as in the previous example. To see this, consider the network $(X,R),$ where $X=\{1,2,3\}$ and $R=\{1\to 2, 2\to 3\}$ and let $G$ be the path on 4 vertices with the initial coloring as in Figure \ref{fig:path example}. Then after one forcing step, the network will terminate with the coloring 1113. 
\begin{figure}[h]\centering
    \begin{tikzpicture}[-latex ,auto ,node distance =2 cm and 1cm ,on grid ,
    semithick ,
    state/.style ={ circle ,top color =white,
    draw,  minimum width =.5cm}]
    \node[state] at (-1,-1/2) (A){$1$};
    \node[state] at (0,-1/2) (B) {$2$};
    \node[state] at (1,-1/2) (C) {$1$};
        \node[state] at (2,-1/2) (D) {$3$};
    \node at (.5,1/3) {$G$};
    \draw [-] (A) -- (B)--(C)--(D); 
    \end{tikzpicture}
    \qquad
        \begin{tikzpicture}[-latex ,auto ,node distance =2 cm and 1cm ,on grid ,
    semithick ,
    state/.style ={ circle ,top color =white,
    draw,  minimum width =.5cm}]
        \node[state] at (-1,-1/2) (A){$1$};
    \node[state] at (0,-1/2) (B) {$1$};
    \node[state] at (1,-1/2) (C) {$1$};
            \node[state] at (2,-1/2) (D) {$3$};
    \draw [-] (A) -- (B)--(C)--(D); 
    \node at (.5,1/3) {$G$ after FS1};
    \end{tikzpicture}
    \caption{Graph $G$ and its termination after one forcing step under $R=\{1\to 2, 2\to 3\}$.}
    \label{fig:path example}
\end{figure}
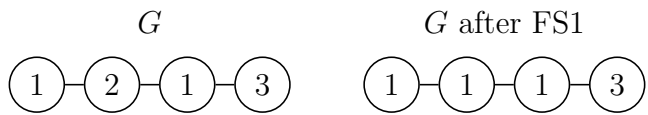

It is known that classical zero forcing on a graph with $n$ vertices terminates after at most $n-1$ color changes \cite{Hogben}.  We now show that all forcing networks terminate whenever forcing with propagation is required.

\begin{theorem} \label{thm: terminates}
Given a forcing network $(X,R)$ and $X-$colored graph, $G$, with $n$ nodes, the network applied to $G$ terminates in at most $n-1$ propagating forcing steps. 
\end{theorem}
\begin{proof}
Let $H_1,\ldots,H_k$ be a partition of $G$ into maximal connected subgraphs whose vertices are all colored with the same color. Note that there exists at most $n$ such subgraphs, so $k\leq n.$ Let $u$ and $v$ be vertices such that $u$ forces $v$ in some propagating forcing step  $\ell.$ If $u \in H_i$ and $v\in H_j$, then after step $\ell$ completes, $V(H_i)$ and $V(H_j)$ will be the same color. This show that the number of maximal connected subgraphs of $G$ whose vertices are all colored with the same color decreases after each propagating forcing step occurs. Therefore, there can be at most $n-1$ propagating forcing steps. 
\end{proof} 

The authors of \cite{Hogben} show that for a tree, $T$, the zero forcing process will terminate in at most $\diam(T)$ forcing steps, where $\diam(G)$ denotes the diameter of a graph $G$ (the maximum distance between any two vertices in $G$). In Example \ref{ex: diam counter}, we give a tree $T$ and a forcing network that terminates in more than $\diam(T)$ propagating forcing steps. However, we prove in Theorem \ref{thm: diam and prop steps}, that for any forcing network and any tree $T$, each propagating forcing step completes in at most $\diam(T)$ forcing steps. 

\begin{example}{\rm \label{ex: diam counter} Let $(X,R)$ be the network with $X=\{1,2,3,4\}$ and  $R=\{1\to 2, 1\to 3, 1\to 4\}$ and let $G$ be the $X-$colored graph shown in Figure \ref{fig: more than diam(G) propagating forcing steps}. Then $\diam(G)=2$ and $(X,R)$ terminates in 3 propagating forcing steps.
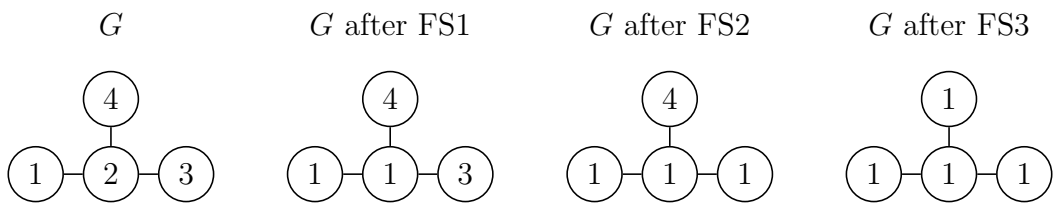
\begin{figure}[ht!]
\centering
\begin{tikzpicture}[-latex ,auto ,node distance =2 cm and 1cm ,on grid ,
semithick ,
state/.style ={ circle ,top color =white,
draw,  minimum width =.5cm}]
\node[state] at (-1,-1) (A){$1$};
\node[state] at (0,-1) (B) {$2$};
\node[state] at (1,-1) (C) {$3$};
\node[state] at (0,0) (D) {$4$};
\node at (0,1) {$G$};
\draw [-] (A) -- (B); 
\draw [-] (B) -- (C);
\draw [-] (B) -- (D);
\end{tikzpicture}
\qquad
\begin{tikzpicture}[-latex ,auto ,node distance =2 cm and 1cm ,on grid ,
semithick ,
state/.style ={ circle ,top color =white,
draw,  minimum width =.5cm}]
\node[state] at (-1,-1) (A){$1$};
\node[state] at (0,-1) (B) {$1$};
\node[state] at (1,-1) (C) {$3$};
\node[state] at (0,0) (D) {$4$};
\node at (0,1) {$G$ after FS1};
\draw [-] (A) -- (B); 
\draw [-] (B) -- (C);
\draw [-] (B) -- (D);
\end{tikzpicture}
\qquad
\begin{tikzpicture}[-latex ,auto ,node distance =2 cm and 1cm ,on grid ,
semithick ,
state/.style ={ circle ,top color =white,
draw,  minimum width =.5cm}]
\node[state] at (-1,-1) (A){$1$};
\node[state] at (0,-1) (B) {$1$};
\node[state] at (1,-1) (C) {$1$};
\node[state] at (0,0) (D) {$4$};
\node at (0,1) {$G$ after FS2};
\draw [-] (A) -- (B); 
\draw [-] (B) -- (C);
\draw [-] (B) -- (D);
\end{tikzpicture}
\qquad
\begin{tikzpicture}[-latex ,auto ,node distance =2 cm and 1cm ,on grid ,
semithick ,
state/.style ={ circle ,top color =white,
draw,  minimum width =.5cm}]
\node[state] at (-1,-1) (A){$1$};
\node[state] at (0,-1) (B) {$1$};
\node[state] at (1,-1) (C) {$1$};
\node[state] at (0,0) (D) {$1$};
\node at (0,1) {$G$ after FS3};
\draw [-] (A) -- (B); 
\draw [-] (B) -- (C);
\draw [-] (B) -- (D);
\end{tikzpicture}
 \caption{An $X-$colored graph, $G$, that terminates in more than $\diam(G)$ propagating forcing~steps.}
 \label{fig: more than diam(G) propagating forcing steps}
\end{figure}
}
\end{example}

\begin{theorem}\label{thm: diam and prop steps} Let $(X,R)$ be a forcing network and let $T$ be an $X-$colored tree. Then each propagating forcing step completes in at most $\diam(T)$ forcing steps. 
\end{theorem}

\begin{proof}
Consider a single propagating forcing step, and suppose that it completes in exactly $k$ forcing steps. Let $v_m$ be a vertex that is forced in step $m$ of this propagating forcing step, and let $v_{m-1}$ be a vertex that forces $v_m$ in step $m$. Since $v_{m-1}$ does not force $v_m$ in a prior step, it must be the case that $v_{m-1}$ is forced in step $m-1$. In this way, we obtain vertices $v_0,\ldots,v_m$ such that for $0\leq i\leq m-1$, $v_i$ forces $v_{i+1}$ in step $i+1.$ Then $(v_0,\ldots,v_m)$ is a path in $T$, and since $T$ is a tree, this shows that the diameter of $T$ is at least $m$. 
\end{proof}

In general, when $G$ is not a tree, a single propagating forcing step may take more than $\diam(G)$ forces to complete, as shown next in Example \ref{ex:prop more than diam}.

\begin{example}\label{ex:prop more than diam}{\rm 
Let $X=\{1,2,3\}, R=\{1\to 2\}$, and $G$ be the graph show in Figure \ref{fig: more than diam(G) step in propagating forcing step}. Then the diameter of $G$ is 2, and the first (and only) propagating forcing step completes in 4 forcing steps.
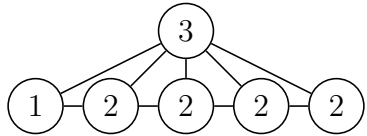
\begin{figure}[ht!]
\centering
\begin{tikzpicture}[-latex ,auto ,node distance =2 cm and 1cm ,on grid ,
semithick ,
state/.style ={ circle ,top color =white,
draw,  minimum width =.5cm}]
\node[state] at (0,0) (A){$3$};
\node[state] at (-2,-1) (B) {$1$};
\node[state] at (-1,-1) (C) {$2$};
\node[state] at (0,-1) (D) {$2$};
\node[state] at (1,-1) (E) {$2$};
\node[state] at (2,-1) (F) {$2$};
\draw [-] (A) -- (B); 
\draw [-] (A) -- (C); 
\draw [-] (A) -- (D); 
\draw [-] (A) -- (E); 
\draw [-] (A) -- (F); 
\draw [-] (B) -- (C);
\draw [-] (C) -- (D);
\draw [-] (D) -- (E); 
\draw [-] (E) -- (F); 
\end{tikzpicture}
 \caption{An $X-$colored graph, $G$, with a propagating step that takes more than $\diam(G)$ forcing steps to complete.}
 \label{fig: more than diam(G) step in propagating forcing step}
\end{figure}
}
\end{example}

\subsection{Termination without propagation}

Example \ref{ex: does not terminate} demonstrates that a forcing network need not terminate when forcing with propagation is not required. We now give a class of forcing networks that terminate even when the propagation condition is not required. 

\begin{theorem}\label{DAG}
Let $(X,R)$ be a forcing network that is acyclic. Then when applied to an $X-$colored graph, $(X,R)$ terminates. 
\end{theorem}
\begin{proof}
Assume that the forcing process does not terminate. Since there are only finitely many ways to color the vertices of $G$ with colors from $X$, then there exist $i,j\geq 0$ with $i+1<j $, such that the coloring of $G$ after forcing step $j$ is exactly the coloring of $G$ after forcing step $i$. Choose any $v\in V(G)$ such that $v$ changes colors in forcing step $i+1$, and let $c_k(v)$ denote the color of $v$ after forcing step $k,$ for $i\leq k\leq j-1.$ Then it follows that $c_{k+1}(v)\to c_k(v)\in R,$ and since $c_i(v)=c_j(v),$ $(X,R)$ contains a cycle. This provides a contradiction since all networks that have topological orderings are acyclic.
\end{proof}
Theorem \ref{DAG} proves that all acyclic forcing networks terminate even without propagation. However, it is still possible for a forcing network that contains a cycle to terminate without propagation. We illustrate this in the next example.

\begin{example}\label{example}{\rm Consider the cyclic forcing
 network $X=\{1,2,3\}, R=\{1\to 2, 2\to 3, 3\to 1\}$ applied to the graph in Figure \ref{fig: triangle}. After 3 forcing steps the forcing network terminates.
 \begin{figure}[h]
    \centering
    \begin{tikzpicture}[-latex ,auto ,node distance =1 cm and 1cm ,on grid ,
    semithick ,
    state/.style ={ circle ,top color =white,
    draw,  minimum width =.3cm}]
    \node[state] at (-1/2,-1/2) (A){$3$};
    \node[state] at (1/2,-1/2) (B) {$2$};
    \node[state] at (0,1/3) (C) {$1$};
    \node at (0,1) {$G$};
    \draw [-] (A) -- (B); 
    \draw [-] (B) -- (C);
    \draw [-] (C) -- (A);
    \end{tikzpicture}
    \qquad
        \begin{tikzpicture}[-latex ,auto ,node distance =1 cm and 1cm ,on grid ,
    semithick ,
    state/.style ={ circle ,top color =white,
    draw,  minimum width =.3cm}]
    \node[state] at (-1/2,-1/2) (A){$3$};
    \node[state] at (1/2,-1/2) (B) {$1$};
    \node[state] at (0,1/3) (C) {$1$};
    \node at (0,1) {$G$ after FS1};
    \draw [-] (A) -- (B); 
    \draw [-] (B) -- (C);
    \draw [-] (C) -- (A);
    \end{tikzpicture}
    \qquad
            \begin{tikzpicture}[-latex ,auto ,node distance =1 cm and 1cm ,on grid ,
    semithick ,
    state/.style ={ circle ,top color =white,
    draw,  minimum width =.3cm}]
    \node[state] at (-1/2,-1/2) (A){$3$};
    \node[state] at (1/2,-1/2) (B) {$1$};
    \node[state] at (0,1/3) (C) {$1$};
    \node at (0,1) {$G$ after FS2};
    \draw [-] (A) -- (B); 
    \draw [-] (B) -- (C);
    \draw [-] (C) -- (A);
    \end{tikzpicture}
    \qquad
            \begin{tikzpicture}[-latex ,auto ,node distance =1 cm and 1cm ,on grid ,
    semithick ,
    state/.style ={ circle ,top color =white,
    draw,  minimum width =.3cm}]
    \node[state] at (-1/2,-1/2) (A){$3$};
    \node[state] at (1/2,-1/2) (B) {$3$};
    \node[state] at (0,1/3) (C) {$3$};
    \node at (0,1) {$G$ after FS3};
    \draw [-] (A) -- (B); 
    \draw [-] (B) -- (C);
    \draw [-] (C) -- (A);
    \end{tikzpicture}
    \caption{Graph along with forcing steps for Example \ref{example}.}
    \label{fig: triangle}
\end{figure}
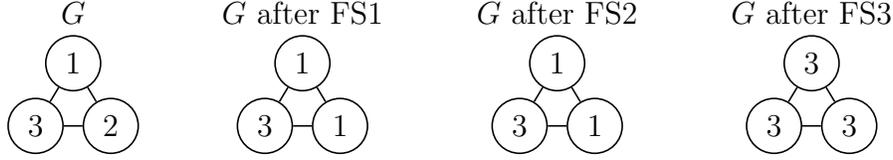
}
\end{example}\subsection{Color-contracted graph technique}\label{sec: Color-Contracted graph technique}
One advantage of forcing with propagation is that we
may reduce a graph $G$ to a smaller graph $G'$ and construct the end state of $G$ from the end state of $G'$. This idea is embedded in the proof of Theorem \ref{thm: terminates}, and we formalize it now.

Let $(X,R)$ be a forcing network and let $G$ be an $X-$colored graph. Partition $G$ into maximal connected subgraphs $H_1,\ldots,H_k$ such for each $i$, the vertices in $V(H_i)$ all share the same color. Note that this partition is unique. Let $G'$ be the graph constructed from $G$ by identifying each $H_i$ as single vertex $v_i$, and $v_i$ and $v_j$ are adjacent in $G'$ if and only if there exist vertices $u\in V(H_i), v\in V(H_j)$ such that $u$ and $v$ are adjacent in $G$. Then we define $G'$ to be the {\em color-contracted graph} of 
$G$. Equivalently, $G'$ is constructed from $G$ by contracting the edges of $G$ whose endpoints have the same color. An example of this is shown in Figure~\ref{fig:new complex graph and its contraction}.

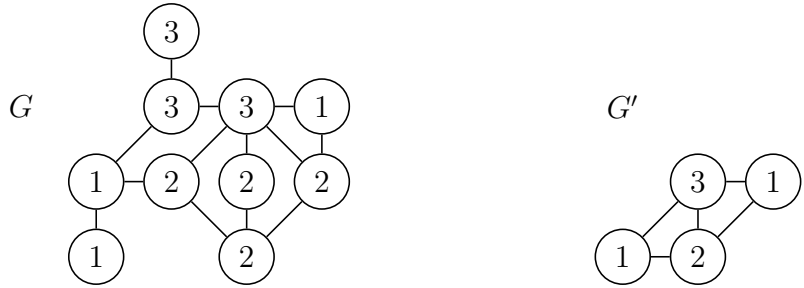
\begin{figure}[h]
\centering
\begin{tikzpicture}[-latex ,auto ,node distance =2 cm and 1cm ,on grid ,
semithick ,
state/.style ={ circle ,top color =white,
draw,  minimum width =.5cm}]
\node[state] at (0,0) (A) {$1$};
\node[state] at (2,0) (B) {$2$};
\node[state] at (0,1) (C) {$1$};
\node[state] at (1,1) (D) {$2$};
\node[state] at (2,1) (E) {$2$};
\node[state] at (3,1) (F) {$2$};
\node[state] at (1,2) (G) {$3$};
\node[state] at (2,2) (H) {$3$};
\node[state] at (3,2) (I) {$1$};
\node[state] at (1,3) (J) {$3$};
\node at (-1,2) {$G$};
\draw[-] (A)--(C);
\draw[-] (D)--(C);
\draw[-] (D)--(B);
\draw[-] (B)--(F);
\draw[-] (B)--(E);
\draw[-] (G)--(C);
\draw[-] (G)--(H);
\draw[-] (H)--(D);
\draw[-] (H)--(E);
\draw[-] (H)--(F);
\draw[-] (H)--(I);
\draw[-] (F)--(I);
\draw[-] (G)--(J);
\node[state] at (7,0) (A1) {$1$};
\node[state] at (8,0) (B1) {$2$};
\node[state] at (8,1) (C1) {$3$};
\node[state] at (9,1) (D1) {$1$};
\node at (7,2) {$G'$};
\draw[-] (A1)--(C1);
\draw[-] (A1)--(B1)--(C1)--(D1)--(B1);
\end{tikzpicture}
\caption{An $X-$colored graph $G$ and its color-contracted graph $G'$.}\label{fig:new complex graph and its contraction}
\end{figure}

\begin{theorem}\label{thm:color-contracted}
Let $(X,R)$ be a forcing network, let $G$ be an $X-$colored graph, and let $G'$ be the color-contracted graph of $G$ as defined above. If forcing with propagation is required, then $\epsilon(\ell_0(G))$ can be constructed from $\epsilon(\ell_0(G'))$ by coloring each vertex in $H_i$ with the color of $v_i$ in $\epsilon(\ell_0(G'))$.
\end{theorem}

\begin{proof}
First note that {by definition}, for all $i$, the set $V(H_i)$ is colored with one color. Also for any $j\geq 0$, note that if one vertex in $V(H_i)$ is forced to a different color by another vertex {at the $j$-th forcing step}, then all vertices in $V(H_i)$ will be forced to become that color since propagation is required.

We now show that after each propagating forcing step, for each $i$, $v_i$ and all vertices in $V(H_i)$ will be the same color. The claim holds true for the initial colorings of $G$ and $G'$ (i.e. after step 0). For step $m\geq 1$, suppose that after step $m-1$, $v_i$ and all vertices in $V(H_i)$ are the same color for each $i$. Let $v_r,v_s\in V(G')$ be such that $v_r$ forces $v_s$ in step $m$, and let $c_r$ and $c_s$ denote the colors of $v_r$ and $v_s$ {after step $m-1$.}
Since $v_r$ and $v_s$ are adjacent in $G'$, by construction of $G'$, it follows that there exist an edge from $V(H_r)$ to $V(H_s)$ in $G$. Furthermore, since $V(H_r)$ is colored $c_r$ and $V(H_s)$ is colored $c_s$ after step $m-1$, then a vertex from $V(H_r)$ will force a vertex in $V(H_s)$ to become color $c_r$ during step $m$, and by propagation, it follows that each vertex in $V(H_s)$ will become color $c_r$ in step $m$. Thus, after step $m$, for each $i$, we have that $v_i$ {in $G'$} and all vertices $V(H_i)$ {in $G$} will be the same color. Now letting $m$ be the final step finishes the proof.
\end{proof}

We demonstrate Theorem \ref{thm:color-contracted} with an example.
\begin{example}{\rm
Consider again the forcing network with $X=\{1,2,3\}$ and $R=\{1\to 2, 2\to 3, 3\to 1\}.$ 
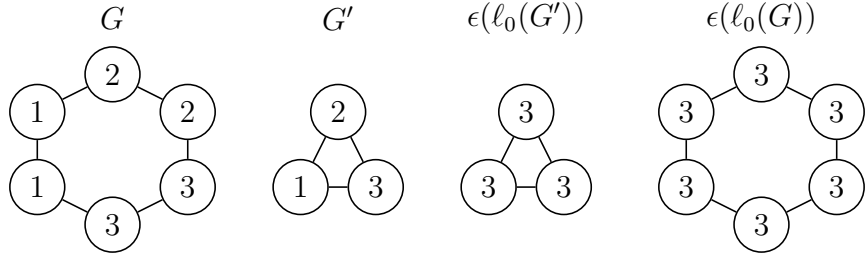
\begin{figure}[ht!]
\centering
    \begin{tikzpicture}[-latex ,auto ,node distance =2 cm and 1cm ,on grid ,
    semithick ,
    state/.style ={ circle ,top color =white,
    draw,  minimum width =.5cm}]
    \node[state] at (-1,-1/2) (A){$1$};
    \node[state] at (0,-1) (B) {$3$};
    \node[state] at (1,-1/2) (C) {$3$};
    \node[state] at (1,1/2) (D) {$2$};
    \node[state] at (-1,1/2) (F) {$1$};
    \node[state] at (0,1) (E) {$2$};
    \node at (0,1.75) {$G$};
    \draw [-] (A) -- (B); 
    \draw [-] (B) -- (C);
    \draw [-] (C) -- (D);
    \draw [-] (D) -- (E);
    \draw [-] (E) -- (F);
    \draw [-] (A) -- (F);
    \node[state] at (2.5,-1/2) (A){$1$};
    \node[state] at (3.5,-1/2) (B) {$3
    $};
    \node[state] at (3,1/2) (C) {$2$};
    \node at (3,1.75) {$G'$};
    \draw [-] (A) -- (B); 
    \draw [-] (B) -- (C);
    \draw [-] (C) -- (A);
    
    \node[state] at (5,-1/2) (A){$3$};
    \node[state] at (6,-1/2) (B) {$3$};
    \node[state] at (5.5,1/2) (C) {$3$};
    \node at (5.5,1.75) {$\epsilon(\ell_0(G'))$};
    \draw [-] (A) -- (B); 
    \draw [-] (B) -- (C);
    \draw [-] (C) -- (A);
    \end{tikzpicture}\qquad
     \begin{tikzpicture}[-latex ,auto ,node distance =2 cm and 1cm ,on grid ,
    semithick ,
    state/.style ={ circle ,top color =white,
    draw,  minimum width =.5cm}]
    \node[state] at (-1,-1/2) (A){$3$};
    \node[state] at (0,-1) (B) {$3$};
    \node[state] at (1,-1/2) (C) {$3$};
    \node[state] at (1,1/2) (D) {$3$};
    \node[state] at (-1,1/2) (F) {$3$};
    \node[state] at (0,1) (E) {$3$};
    \node at (0,1.75){ $\epsilon(\ell_0(G))$};
    \draw [-] (A) -- (B); 
    \draw [-] (B) -- (C);
    \draw [-] (C) -- (D);
    \draw [-] (D) -- (E);
    \draw [-] (E) -- (F);
    \draw [-] (A) -- (F);
    \end{tikzpicture}
    \caption{A demonstration of Theorem \ref{thm:color-contracted}.}

\end{figure}
}
\end{example}

In Section \ref{sec:families}, we show that the color-contracted graph is an useful tool for analyzing end~states. 

\section{The 3-cyclic forcing network}\label{section:cyclic}

One may view classical zero forcing   as the forcing network $(X,R)$ where  $X=\{1,2\}$ and $R=\{1\to 2\}.$ One natural extension of zero forcing   is to consider forcing networks $(X,R)$ with multiple colors such that for any two colors $i,j$ in $X$, there is a relation between them (meaning either $i \to j \in R$ or $j \to i \in R).$ As a first step, we consider such forcing networks with only 3 colors. 

For $X=\{1,2,3\}$, there are only two forcing networks up to isomorphism: $R_1=\{1\to 2, 2\to 3, 1\to 3\},$ and $R_2=\{1\to 2, 2\to 3, 3\to 1\}$. It is not hard to see that for the network $(X, R_1)$, if $G$ is an $X-$colored graph such that at least one the vertices is colored 1, then forcing  will terminate with all vertices colored 1 (since neither vertices colored 2 nor vertices colored 3 can force vertices colored 1). Otherwise, forcing  will terminate with all vertices colored 2 (if there is at least one vertex colored 2), or all vertices colored 3 (if there are no vertices colored 1 and no vertices colored 2 initially).

On the other hand, the behavior of the network $(X,R_2)$ is much more interesting and far less predictable, so we formally define this network and study it in this section. 
\begin{defn}
{\rm The {\em $3$-cyclic forcing network} is the network $(X,R)$ where $X=\{1,2,3\}$ and $R=\{1\rightarrow 2, 2\rightarrow 3, 3\rightarrow 1\}$.}
\end{defn}

 \begin{lemma}\label{lem:one color end state}
 Let $(X,R)$ be the 3-cyclic forcing network and let $G$ be an $X-$colored graph. Then $\epsilon(\ell_0(G))=\vec{i}$ for some $i\in \{1,2,3\}$ when forcing with propagation occurs.
\end{lemma}

\begin{proof}
 If two adjacent vertices are colored with different colors, then a force can be applied. By Theorem \ref{thm: terminates}, the forcing network terminates, so it must be the case that $\epsilon(\ell_0(G))=\vec{i}$ for some $i\in \{1,2,3\}$.
\end{proof}

When only two of the three colors are used in the initial coloring, the end state can be characterized completely. This is summarized and stated without proof in Lemma \ref{end states for two colors} as it can be easily verified.
\begin{lemma}\label{end states for two colors} 
Let $(X,R)$ be the 3-cyclic forcing network and let be $G$ be an $X-$colored graph. Then for any propagating forcing step $i$,
\begin{itemize}
    \item[(1)] If $\ell_i(G)$ consists only of colors 1 and 2, then $\epsilon(\ell_0(G))=\vec{1}$.
    \item[(2)] If $\ell_i(G)$ consists only of colors 2 and 3, then $\epsilon(\ell_0(G))=\vec{2}$.
     \item[(3)] If $\ell_i(G)$ consists only of colors 1 and 3, then $\epsilon(\ell_0(G))=\vec{3}$.
    
\end{itemize}
\end{lemma}

We note that the analysis of end states of the 3-cyclic forcing network for graphs whose initial coloring consists of each of the three colors is highly sensitive to small changes in the initial coloring, and are therefore more complicated. We demonstrate this sensitivity in Example \ref{analysis is sensitive}.

\begin{example}\label{analysis is sensitive}{\rm
    Letting $\ell_0(P_7)=2312321$, the forcing procedure is as follows:
    \begin{center}
        $2312321\xrightarrow{\text{PFS1}}2311311\xrightarrow{\text{PFS2}}2211311\xrightarrow{\text{PFS3}}2233333\xrightarrow{\text{PFS4}}2222222.$
    \end{center}
    However, by letting  $\ell_0(P_7)=2312322$, the forcing procedure becomes:
    \begin{center}
        $2312322\xrightarrow{\text{PFS1}}2311322\xrightarrow{\text{PFS2}}2211222\xrightarrow{\text{PFS3}}1111111.$
    \end{center}
    Although the initial colorings only differ by one color, the end states are different.
    }
\end{example}

We now partially characterize end states based on initial colorings when each of the three colors are used at least once. It follows from Theorem \ref{thm:color-contracted} and Lemma \ref{lem:one color end state} that if $G'$ is the color-contracted graph of $G$, then $\epsilon(\ell_0(G))=\epsilon(\ell_0(G'))$, so in the following lemma, we assume that $G$ is color-contracted.

\begin{lemma}\label{end states with three colors}
Let $(X,R)$ be the 3-cyclic network and let $G$ be an $X-$colored graph that is color-contracted such that each color appears at least once.
\begin{itemize}
\item[(1)] If each vertex colored 3 has a neighboring vertex $v$ colored 2 such that all neighbors of $v$ are colored 3, then $
\epsilon(\ell_0(G))=\vec{1}.$
\item[(2)] If each vertex colored 2 has a neighboring vertex $v$ colored 1, then $
\epsilon(\ell_0(G))=\vec{3}.$
\item[(3)] If each vertex colored 1 has a neighboring vertex $v$ colored 3 such that all neighbors of $v$ are colored 1 and case (2) does not hold, then $
\epsilon(\ell_0(G))=\vec{2}.$
\end{itemize}

\end{lemma}

\begin{proof}
For the first statement, we show that after at most two propagating steps, there will be no vertices colored 3, and the claim follows from Lemma \ref{end states for two colors}. Let $u$ be a vertex that is colored 3 and pick $v$ from the neighborhood of $u$ such that $v$ is colored 2 and all neighbors of $v$ are colored 3. Since $v$ has no neighbors colored $1$ or $2$, then the color of $v$ will not change to $1$ during the first propagating step. During the second propagating step (or the first if no vertices colored 1 are able to force initially), $v$ will force $u$ to become 2, so no vertices colored 3 will remain, and the statement follows from Lemma \ref{end states for two colors}

For the second statement, note that there will be no vertices colored 2 after the first propagating step. The claims then follow from Lemma~\ref{end states for two colors}.

For the third statement, consider the step in which vertices colored 3 force for the first time. We show that any vertex that is colored 1 will be colored 3 during this propagating step: Any vertex colored 1 that was originally colored 1 is adjacent to a vertex colored 3, so it will be colored 3 in this step. For any vertex $w$ that was forced to become 1 during propagating step one, there must exist a path $(w,v_1,...,v_k,u)$ where each vertex on the path is colored 1 and vertex $u$ was originally colored 1. Thus, during this propagating step, $u$ will be colored 3, and by propagation, vertices $v_k,...,v_1,w $ will be colored 3. This shows that there will be no vertices colored 1 after this propagating step. By hypothesis, there is at least one vertex originally colored 2 that is still colored 2. It follows from Lemma~\ref{end states for two colors} that $\epsilon(\ell_0(G))=\vec{2}.$
\end{proof}

We continue the study of the $3$-cyclic network in the next section by specializing to certain families of graphs. 

\subsection{Results of the 3-cyclic network  on special families of graphs}\label{sec:families}
In this section, we provide results on ends states for complete graphs, complete bipartite graphs, and for some path graphs. Much work remains to be done on classifying all path graphs, as we only study those whose color-contracted graph has length at most 5. We demonstrate that even when a small number of vertices is considered, the analysis of the path can be very complicated.

\begin{lemma}\label{lemma:completegraph}
Let $(X,R)$ be the $3$-cyclic network and let the complete graph $K_n$ be colored such that each color appears at least once in $\ell_0(K_n)$. Then $\epsilon(\ell_0(K_n))=\vec{3}$.
\end{lemma}

\begin{proof}
The proof follows directly from the third statement of Lemma \ref{end states with three colors}.
\end{proof}

\begin{lemma}\label{cbexception}
Let $(X,R)$ be the $3$-cyclic network and let $K_{m,n}$ be the complete bipartite graph with vertex set $V(K_{m,n})=A\cup B$, where $A$ and $B$ are such that no edge has both endpoints in the same subset. If $K_{m,n}$ is colored such that each color appears at least once in $\ell_0(K_{m,n})$ and either each vertex in $A$ is colored 3 or each vertex in $B$ is colored 3, then $\epsilon(\ell_0(K_{m,n}))=\vec{1}$. Otherwise, $\epsilon(\ell_0(K_{m,n}))=\vec{3}$.
\end{lemma}
\begin{proof}
{Let $\ell_i(A)$ and $\ell_i(B)$ denote the colors appearing in the vertices of $A$ and $B$ after the  $i$-th propagating step of the 3-color cyclic change rule. Then, the following cases exhaust all of the possibilities for initial colorings of the vertices in $A$ and $B$ (up to symmetry):
\begin{enumerate}
\item If $\ell_0(A)=\{1\}$ and $\ell_0(B)=\{2,3\}$ or $\{1,2,3\}$, then $\ell_1(A)=\{1\}$ and $\ell_1(B)=\{1,3\}$. Applying Lemma \ref{end states with three colors} shows $\epsilon(\ell_0(K_{m,n}))=\vec{3}$.
\item  If $\ell_0(A)=\{2\}$ and $\ell_0(B)=\{1,3\}$ or $\{1,2,3\}$, then $\ell_1(A)=\{1\}$ and $\ell_1(B)=\{1,3\}$. Applying Lemma \ref{end states with three colors} shows $\epsilon(\ell_0(K_{m,n}))=\vec{3}$.
\item  If $\ell_0(A)=\{3\}$ and $\ell_0(B)=\{1,2\}$ or $\{1,2,3\}$, , then 
$\ell_1(A)=\{2\}$ and $\ell_1(B)=\{1,2\}$. Applying Lemma \ref{end states with three colors} shows $\epsilon(\ell_0(K_{m,n}))=\vec{1}$.
\item If $\ell_0(A)=\{1,2\}$ and $\ell_0(B)=\{1,3\}$, $\{2,3\}$, or $\{1,2,3\}$ then $\ell_1(A)=\{1\}$ and $\ell_1(B)=\{1,3\}$. Applying Lemma \ref{end states with three colors} shows $\epsilon(\ell_0(K_{m,n}))=\vec{3}$.
\item If $\ell_0(A)=\{1,3\}$ and $\ell_0(B)=\{2,3\}$ or $\{1,2,3\}$ then $\ell_1(A)=\{1,3\}=\ell_1(B)$. Applying Lemma \ref{end states with three colors} shows $\epsilon(\ell_0(K_{m,n}))=\vec{3}$.
\item If $\ell_0(A)=\{1,2,3\}$ and $\ell_0(B)=\{1,2,3\}$  or $\{2,3\},$ then $\ell_1(A)=\{1,3\}=\ell_1(B)$. Applying Lemma \ref{end states with three colors} shows $\epsilon(\ell_0(K_{m,n}))=\vec{3}$.\qedhere
\end{enumerate}}
\end{proof}

We now focus our attention on classifying the end states of some initial colorings of the path $P_n$ that contain all three numbers.

\begin{theorem}
Let $P_n$ be an $X-$colored graph, and let $P_n'$ be the color-contracted graph of $P_n$. Then $\epsilon(\ell_0(P_n))=\epsilon(\ell_0(P_n')).$
\end{theorem}

\begin{proof}
By Lemma \ref{lem:one color end state}, $\epsilon(\ell_0(P_n'))=i$ for some $\vec{i}\in \{1,2,3\}$, and when constructing $\epsilon(\ell_0(P_n))$ from $\epsilon(\ell_0(P_n'))$ using the technique of Theorem \ref{thm:color-contracted}, we obtain that $\epsilon(\ell_0(P_n))=\vec{i}.$
\end{proof}

\begin{theorem}\label{lemma:pathall2}
    Let $(X,R)$ be the 3-cyclic network, and suppose $P_n=(a_1a_2\cdots a_{n})$ is an $X-$colored path such that each color appears at least once. Let $P_n'=(b_1b_2\cdots b_{k})$ be the color-contracted graph of $P_n$ and suppose that $k\leq 5.$ 
    {\rm
    \begin{enumerate}
        \item[(1)] Suppose the sequence 323 appears in $\ell_0(P_n')$. If no other 3 appears, then $\epsilon(\ell_0(P_n'))=\vec{1}$; otherwise the coloring $\epsilon(\ell_0(P_n'))=\vec{2}$.
        \item[(2)] Suppose the sequence 131 appears in $\ell_0(P_n')$. If no other 1 appears, then $\epsilon(\ell_0(P_n'))=\vec{2}$; otherwise the coloring $\epsilon(\ell_0(P_n'))=\vec{3}$.
        \item[(3)] If each vertex colored 2 is adjacent to a vertex colored 1 in $\ell_0(P_n')$, then $\epsilon(\ell_0(P_n'))=\vec{3}$.
        \item[(4)] Suppose $\ell_0(P_n')$ begins with 231 or ends 132. If no additional 3 appears, then $\epsilon(\ell_0(P_n'))=\vec{1}.$ Otherwise, $\epsilon(\ell_0(P_n'))=\vec{2}$, or $P_n'$ has the coloring $23132$, in which case $\epsilon(\ell_0(P_n'))=\vec{1}$.
        \item[(5)] If neither (1)-(4) holds true, then $\epsilon(\ell_0(P_n'))=\vec{1}$, or $\ell_0(P_n')=23213$ or $31232$, in which case, $\epsilon(\ell_0(P_n'))=\vec{2}$.
     \end{enumerate}}
\end{theorem}

\begin{proof}
We proceed via a case-by case analysis.
\begin{enumerate}
\item [(1)] If the sequence 323 appears and there is no additional 3, then by (1) of  Lemma \ref{end states with three colors}, $\epsilon(\ell_0(P_n'))=\vec{1}$. Otherwise $P_n'$ must have the coloring 32313 or 31323, and in each case $\epsilon(\ell_0(P_n'))=\vec{2}.$

\item [(2)] If the sequence 131 appears and there is no additional 1, then by (2) of  Lemma \ref{end states with three colors}, $\epsilon(\ell_0(P_n'))=\vec{2}$. Otherwise $P_n'$ must have the coloring 13121 or 12131, and in each case $\epsilon(\ell_0(P_n'))=\vec{3}.$

\item [(3)] If each vertex colored 2 is adjacent to a vertex colored 1, by (3) of Lemma \ref{end states with three colors}, $\epsilon(\ell_0(P_n'))=\vec{3}.$

\item [(4)] Suppose the coloring begins with 231 or ends 132. If no additional 3 appears, then after two propagating time steps, each vertex will be colored 1 or 2, so  $\epsilon(\ell_0(P_n'))=\vec{1}.$ If there does exist an additional vertex colored 3 and the coloring is not 23132, then the coloring is either 2313, 23131, 32123, 23123, or the reverse of one of these colorings, and it can be quickly verified that $\epsilon(\ell_0(P_n'))=\vec{2}.$

\item [(5)] Suppose now that neither (1), (3), nor (4) holds true. Since (3) fails, then there must exist a vertex colored 2 that is adjacent to no vertex colored 1, and since (1) fails, then this vertex must be an end point. It follows that the coloring either starts in 23 or ends in 32. Since (4) fails, then the coloring either starts with 232 or ends with 232 (which shows that (2) fails).  Since 1 must appear in the coloring and (1) fails, then the coloring either starts with 2321 or ends with 1232. If the coloring is 23212 or 21232, it is easy to verify that $\epsilon(P_n'))=\vec{1}.$ Otherwise, the coloring is 2321, 23213, or the reverse of one of these, and in either case, we see that $\epsilon(\ell_0(P_n')=\vec{1}.$\qedhere
\end{enumerate}
\end{proof}

If the color-contracted graph has more than five vertices, then there are a large number of cases, similar to the ones enumerated in Theorem~\ref{lemma:pathall2}, that need to be considered. Until one can find more general sufficient conditions for obtaining a particular end state, it remains difficult to find a statement analogous to Theorem~\ref{lemma:pathall2} in the case of larger color-contracted graphs.

\section{Future work}\label{sec:future}

In classical zero forcing, it is often of interest to determine how many forcing steps are required to terminate the process, given an initial coloring; this is called {\em propagation time}. It is known that a graph on $n$ vertices has propagation time at most $n-1$, with the path $P_n$ being the only graph to achieve the upper bound \cite{AIMMINIMUMRANKSPECIALGRAPHSWORKGROUP20081628}.  We show in Theorem \ref{thm: terminates} that when a forcing network is applied to a graph on $n$ vertices, the process will terminate in at most $n-1$ propagating time steps. This gives rise to the following question:
\begin{question} Given a network $(X,R)$, is it possible to classify the $X-$colored graphs on $n$ vertices for which the network terminates in exactly $n-1$ propagating time step?
\end{question}
{One could also study the effect elementary graph operations have on the propagation time.}
\begin{question} How do different graph operations, {such as addition or deletion of a vertex or of an edge, merging and splitting of vertices, and edge contraction,} affect the end state and total number of forcing steps {in the multi-color forcing procedure?}
\end{question}

Lastly, it would also be of interest to characterize graphs for which a multi-color forcing procedure without propagation actually terminates. We state this precisely below. 

\begin{question}
Given a forcing network $(X,R)$ whose steps do not allow propagation, can one give sufficient and necessary conditions on an $X-$colored graph, $G$, such that the multi-color forcing procedure will terminate?
\end{question}
\section*{Acknowledgments}
The authors thank Minerva Catral for introducing them to this problem.

\end{document}